\documentclass{amsart}
\usepackage{amssymb,
enumitem,
mathrsfs,
mathtools,
hyperref,
tikz,
upgreek,
verbatim,
scalerel}
\usepackage[mathscr]{euscript}
\usetikzlibrary{matrix}

\definecolor{granata}{RGB}{131,29,28}
\hypersetup{colorlinks=true, allcolors=granata}
\numberwithin{equation}{section}
\usepackage[T1]{fontenc}
\usepackage[shadow]{todonotes}

\newcommand{\pre}[2]{{}^{#1} #2}
\newcommand{\set}[1]{\{ #1 \}} 
\newcommand{\setm}[2]{\{ #1 \mid #2 \}} 

\newcommand{\pow}{\mathcal{P}}

\newcommand{\dom}{\operatorname{dom}}

\renewcommand{\phi}{\varphi}

\newcommand{\OCP}{{\rm OCP}}

\newcommand{\embeds}{\sqsubseteq}
\newcommand{\analytic}{\boldsymbol{\Sigma}_1^1}

\newcommand{\Arw}[1]{\mathrm{Arw}(#1)}
\newcommand{\Hom}[2]{\mathrm{Hom}(#1,#2)}
\DeclareMathOperator{\End}{End}
\DeclareMathOperator{\Ima}{Im}

\newcommand\reallywidehat[1]{\arraycolsep=0pt\relax%
\begin{array}{c}
\stretchto{
  \scaleto{
    \scalerel*[\widthof{\ensuremath{#1}}]{\kern-.5pt\bigwedge\kern-.5pt}
    {\rule[-\textheight/2]{1ex}{\textheight}} 
  }{\textheight} %
}{1ex}\\           
#1\\                 
\rule{-1ex}{0ex}
\end{array}
}

\newenvironment{enumerate-(a)}{\begin{enumerate}[label={\upshape (\alph*)}, leftmargin=2pc]}{\end{enumerate}}
\newenvironment{enumerate-(a)-r}{\begin{enumerate}[label={\upshape (\alph*)}, leftmargin=2pc,resume]}{\end{enumerate}}
\newenvironment{enumerate-(a)-5}{\begin{enumerate}[label={\upshape (\alph*)}, leftmargin=2pc,start=5]}{\end{enumerate}}
\newenvironment{enumerate-(A)}{\begin{enumerate}[label={\upshape (\Alph*)}, leftmargin=2pc]}{\end{enumerate}}
\newenvironment{enumerate-(A)-r}{\begin{enumerate}[label={\upshape (\Alph*)}, leftmargin=2pc,resume]}{\end{enumerate}}
\newenvironment{enumerate-(i)}{\begin{enumerate}[label={\upshape (\roman*)}, leftmargin=2pc]}{\end{enumerate}}
\newenvironment{enumerate-(i)-r}{\begin{enumerate}[label={\upshape (\roman*)}, leftmargin=2pc,resume]}{\end{enumerate}}
\newenvironment{enumerate-(I)}{\begin{enumerate}[label={\upshape (\Roman*)}, leftmargin=2pc]}{\end{enumerate}}
\newenvironment{enumerate-(I)-r}{\begin{enumerate}[label={\upshape (\Roman*)}, leftmargin=2pc,resume]}{\end{enumerate}}
\newenvironment{enumerate-(1)}{\begin{enumerate}[label={\upshape (\arabic*)}, leftmargin=2pc]}{\end{enumerate}}
\newenvironment{enumerate-(1)-r}{\begin{enumerate}[label={\upshape (\arabic*)}, leftmargin=2pc,resume]}{\end{enumerate}}
\newenvironment{itemizenew}{\begin{itemize}[leftmargin=2pc]}{\end{itemize}}

\newtheorem{theorem}{Theorem}[section]
\newtheorem{lemma}[theorem]{Lemma}

\newtheorem{corollary}[theorem]{Corollary}
\newtheorem{proposition}[theorem]{Proposition}
\newtheorem{question}[theorem]{Question}

\theoremstyle{definition}
\newtheorem{definition}[theorem]{Definition}
\newtheorem{notation}[theorem]{Notation}
\theoremstyle{remark}
\newtheorem{remark}[theorem]{Remark}

\begin{document}

\title[Embeddability on torsion-free abelian groups of uncountable size]{The
complexity of the embeddability relation between  torsion-free abelian groups
of uncountable size}
\date{\today}
\author[F.~Calderoni]{Filippo Calderoni}

\address{Dipartimento di matematica \guillemotleft{Giuseppe Peano}\guillemotright, Universit\`a di Torino, Via Carlo Alberto 10, 10121 Torino --- Italy}
\email{filippo.calderoni@unito.it}

\thanks{The author is very grateful to Adam J. Prze\'zdziecki for explaining some crucial parts of \cite{Prz14} and for some invaluable discussions which eventually lead to the proof of the main results of this paper. The author would like to thank Tapani Hyttinen for some useful comments that contributed to the last section. 
Moreover, the author thanks Andrew Brooke-Taylor, Rapha\" el Carroy, Simon Thomas, and Matteo Viale for comments and interesting discussions.
The results presented in this article will be part of the PhD thesis of the author and carried out under the supervision of Luca Motto Ros. 
This work was supported by the ``National Group for the Algebraic and Geometric Structures and their Applications'' (GNSAGA--INDAM)}
\begin{abstract} 
We prove that for every uncountable cardinal \( \kappa \) such that \(\kappa^{<\kappa}=\kappa\), the quasi-order of embeddability on the \(\kappa\)-space of \( \kappa \)-sized graphs Borel reduces to the embeddability
on the \( \kappa\)-space of \( \kappa \)-sized torsion-free abelian groups. Then we use the same techniques to prove that
the former Borel reduces to the embeddability on the \(\kappa\)-space of \( \kappa \)-sized \( R \)-modules,
 for every \(\mathbb{S}\)-cotorsion-free ring \( R \) of cardinality less than the continuum.
As a consequence we get that all the previous are complete \( \analytic \) quasi-order.
\end{abstract}
\maketitle


\section{Introduction}

A subset of a topological space is \emph{\( \kappa \)-Borel} if it is in the smallest \( \kappa \)-algebra containing the open sets.
Given two spaces $X,Y$, a function $f\colon X\to Y$ is \emph{ \(\kappa\)-Borel (measurable)} if the preimage through $f$ of every open subset of $Y$ is \( \kappa \)-Borel.
Two spaces $X,Y$ are said \emph{\( \kappa \)-Borel isomorphic} if there is a \( \kappa \)-Borel bijection $X\to Y$ whose inverse is \( \kappa \)-Borel too. When \(\kappa=\aleph_{1}\) these notions coincide with the ones of  Borel sets, Borel functions, Borel isomorphism (see the classical reference~\cite{Kec}). More background details and examples will be given in the next section.

Let \(\kappa \) be an infinite cardinal. A topological space $X$ is a \emph{\( \kappa \)-space} if it admits a basis of size $\mathop{\leq}\kappa$.
We denote by \(\pre{\kappa}{\kappa}\) the \emph{generalized Baire space}; i.e., the set of functions from \(\kappa\) to itself
endowed with the topology generated by the sets of those functions extending a
fixed function from a bounded subset of \(\kappa\) to \(\kappa\).
We assume the hypothesis \(\kappa^{<\kappa}=\kappa\)\ , which implies that \(\pre{\kappa}{\kappa}\) is a \( \kappa \)-space.
A \( \kappa \)-space is \emph{standard Borel} if it is $\kappa^+$-Borel isomorphic to a $\kappa^+$-Borel subset of $\pre{\kappa}{\kappa}$.
If $X$ is a standard Borel \( \kappa \)-space, we say that \( A\subseteq X \) is \emph{\( \kappa \)-analytic (or $\analytic$)} if it is a continuous image of a closed subset of $\pre{\kappa}{\kappa}$.
The set of
\( \kappa \)-analytic subsets of $X$ is usually denoted by $\analytic(X)$. We are interested in \emph{quasi-orders} (i.e., reflexive and transitive binary relations) defined over standard Borel \( \kappa \)-spaces.
A quasi-order \( Q \) on \( X \) is \emph{analytic}, or \( \analytic \), if and only if \( Q\in\analytic(X\times X)\).
Let \( P ,Q \) be \( \analytic \) quasi-orders on the standard Borel \(\kappa\)-spaces \( X \) and \( Y \), respectively. We say that \( P \) \emph{Borel reduces} to \( Q \) if there is a \(\kappa^{+}\)-Borel function \( f\colon X\to Y\) such that
\(
\forall x_0x_1\in X\ (x_0\mathrel{P}x_1\Leftrightarrow {f(x_0)}\mathrel{Q}{f(x_1)}).
\)

\begin{theorem}[essentially Williams~\cite{Wil14}]\label{Theorem : Wil}
For every infinite cardinal \( \kappa \) such that \( \kappa^{<\kappa}=\kappa \), the  embeddability quasi-order on the \(\kappa\)-space of \( \kappa \)-sized graphs Borel reduces to the embeddability
on the \(\kappa\)-space of \( \kappa \)-sized groups.
\end{theorem}

The above result was proved in~\cite[Theorem 5.1]{Wil14} for \( \kappa=\omega \), and the same proof works for \( \kappa \) uncountable as well. In view of a result of Louveau and Rosendal (see \cite[Theorem 3.1]{LouRos}), Theorem~\ref{Theorem : Wil} in case \(\kappa=\omega\) yields that
the embeddability relation on countable groups is a complete \( \analytic \) quasi-order (i.e., a maximum among all \(\analytic\) quasi-orders up to Borel reducibility).
To prove Theorem~\ref{Theorem : Wil}, Williams maps every countable graph to a group generated by the vertices of the graph and
satisfying some small cancellation hypothesis.
Such groups are not abelian and have many torsion elements that are used to encode the edge relation of the corresponding graphs.
So one may wonder whether there exists another Borel reduction from embeddability between countable graphs to embeddability on countable torsion-free abelian groups.

At about the same time as Theorem~\ref{Theorem : Wil} was proved, Prze\'zdziecki showed in \cite{Prz14} that the category of graphs almost-fully embeds into the category of abelian groups. I.e., there exists a functor \( G\colon\mathscr{G}raphs\to \mathscr{A}b \) such that for every two graphs \( T,V\) there is a natural isomorphism
\[
\mathbb{Z}[\Hom{T}{V}] \cong \Hom{GT}{GV},
\]
where \( \mathbb{Z}[B] \) is defined as the free abelian group with basis \( B \). A closer look into the construction of the functor
reveals that it takes values in the subcategory of torsion-free abelian groups. Unfortunately for us, the restriction of \(G\) to the standard Borel space of countable graphs is not a Borel reduction in the classical sense because countable graphs are
sent to groups of size the continuum.

In this paper we work in the framework of generalized descriptive set theory. By tweaking the construction of \cite{Prz14} we show the following.
\begin{theorem}\label{Theorem : main}
 For every uncountable \( \kappa \) such that \( \kappa^{<\kappa}=\kappa \)\ , there is a 
Borel
reduction from the quasi-order of embeddability 
on the \(\kappa\)-space of graphs of size \( \kappa \) to the embeddability on the \(\kappa\)-space of \(\kappa\)-sized torsion-free abelian groups.
\end{theorem}
So,
in view of the results of \cite{Mot13} and \cite{MilMot}, the embeddability relation between \(\kappa\)-sized torsion-free abelian groups is a complete
\(\analytic\) quasi-order. 

In Section 2 we introduce the main definitions of Borel reducibility in the framework of generalized descriptive set theory. In Section 3 we recall some theorems on the existence of \( R \)-modules with prescribed endomorphism ring. Such theorems
will be used to define the reductions we present in the ensuing sections. Section 4 is dedicated to the proof of Theorem~\ref{Theorem : main}: we define a Borel reduction from the embeddability relation on \( \kappa \)-sized graphs to embeddability on \( \kappa \)-sized torsion-free abelian groups. 
In Section 5 we follow the main ideas of \cite{GobPrz}, and we exploit the techniques used in Section 4 to prove an analogue of Theorem~\ref{Theorem : main} for the embeddability relation on \( R \)-modules, for every \( \mathbb{S} \)-cotorsion-free ring \( R \) of cardinality less than the continuum.

\begin{theorem}\label{Theorem : R-mod} Let \( R \) be a commutative \( \mathbb{S} \)-cotorsion-free ring of cardinality less than the continuum. For every uncountable \( \kappa \) such that \( \kappa^{<\kappa}=\kappa \)\ ,
 there is a 
Borel
reduction from
the embeddability quasi-order on the  \(\kappa\)-space of graphs of size \( \kappa \) to embeddability on the \(\kappa\)-space of \(\kappa\)-sized \(R\)-modules.
\end{theorem}

In Section 6 we address the problem of determining the Borel complexity of
\( \cong^{\kappa}_\mathsf{TFA} \) the isomorphism on \( \kappa \)-sized torsion-free abelian groups. We point out 
that a result of \cite{HytMor} implies the following.

\begin{theorem}\label{Theorem : iso}
Assume that \( \mathrm V=\mathrm L \) and \( \kappa \) is inaccessible. Then
the isomorphism relation 
on \( \kappa \)-sized torsion-free abelian groups
is a
complete \( \analytic \) equivalence relation.
\end{theorem}

\section{Preliminaries}
We consider the \emph{generalized Baire space} $\pre{\kappa}{\kappa}\coloneqq\setm{x}{x\colon\kappa\to\kappa}$ for an uncountable cardinal \( \kappa \).
Unless otherwise specified, $\pre{\kappa}{\kappa}$ is endowed with the \emph{bounded topology} \( \tau_b \), i.e., the topology generated by the basic open sets
\[
N_s=\setm{x\in{\pre{\kappa}{\kappa}}}{x\supseteq s},
\]
where \( s\in\pre{<\kappa}{\kappa} \).
Notice that since \( \kappa \) is uncountable, the topology $\tau_b$ is strictly finer than the product topology. We assume the hypothesis
\begin{equation}\label{equation : k<kequalsk}
\kappa^{<\kappa}=\kappa,
\end{equation}
consequently we get that $\pre{\kappa}{\kappa}$ is a \( \kappa \)-space, as the basis
\(\setm{N_{s}}{s\in\pre{<\kappa}{\kappa}}\) has size \(\kappa\).
 The \emph{generalized Cantor space} $\pre{\kappa}{2}\coloneqq\setm{x\in\pre{\kappa}{\kappa}}{x\colon\kappa\to 2}$ is a closed subset of $\pre{\kappa}{\kappa}$ and therefore it is standard Borel with the relative topology.
We recall the following proposition which gives some characterizations of \( \kappa \)-analytic sets.
\begin{proposition}\label{Prop : analytic}
Let $X$ be a standard Borel \( \kappa \)-space and  $A \subseteq X$ nonempty. Then, the following are equivalent:
\begin{enumerate-(i)}
\item  $A$ is \( \kappa \)-analytic;
\item $A$ is a continuous image of some \(\kappa^{+}\)-Borel \( B\subseteq \pre{\kappa}{\kappa} \);
\item $A$ is a \(\kappa^{+}\)-Borel image of some \(\kappa^{+}\)-Borel \( B\subseteq \pre{\kappa}{\kappa} \);
\item \label{cond : 4ana} $A$ is the projection $p(F) = \setm{x\in X}{\exists y \in \pre{\kappa}{\kappa}\ ((x,y)\in F)}$ of some closed subset $F \subseteq X\times\pre{\kappa}{\kappa}$.
\end{enumerate-(i)}
\end{proposition}

A proof of Proposition \ref{Prop : analytic} is given in \cite[Section 3]{Mot13}.
It is specially worth to note that in view of \ref{cond : 4ana} we are allowed to use a generalization of the celebrated Tarski-Kuratowski algorithm (see~\cite[Appendix~C]{Kec}).
That is, a set \( A \subseteq \pre{\kappa}{\kappa}\) is \( \kappa \)-analytic if it is defined
by an expression involving only \( \kappa \)-analytic sets, atomic connectives,
\( \exists \alpha, \forall \alpha \) (where \( \alpha \) varies over a set of cardinality
\( \leq \kappa\)), and existential quantification over a standard Borel \( \kappa \)-space.

In the remainder of this paper we fix some uncountable \( \kappa \) and study standard Borel \( \kappa \)-spaces with the assumption \( \kappa^{<\kappa}=\kappa \). 
For ease of exposition we simply say Borel and analytic instead of saying respectively \( \kappa^+ \)-Borel and \( \kappa \)-analytic, whenever \( \kappa \) is clear from the context.

\subsection{Spaces of \( \kappa \)-sized structures}

In this subsection we recall briefly how to define the standard Borel \( \kappa \)-spaces of uncountable structures of size \( \kappa \). While in~\cite{AndMot,FriHytKul,Mot13}
the authors are concerned only with countable languages, we
extend the basic definitions to uncountable ones.
Our approach is motivated by the aim to
develop a unique framework to treat algebraic objects
with the most diverse features.
For example, following our approach it is possible to define the standard Borel \( \kappa \)-space of \( \kappa \)-sized \( R \)-modules, for any fixed ring \( R \) with \( |R|<\kappa \). Such spaces will be taken into account in Section \ref{Section5}.

\begin{definition}\label{Fact : bounded topology}
If \( A\) is a set of size \( \kappa \), then any bijection \( f\colon\kappa \to A \) induces a bijection from \( \pre{\kappa}{\kappa} \) to \( \pre{A}{\kappa} \), so that the \emph{bounded topology} can be copied on \( \pre{A}{\kappa} \). A basis for such topologies is given by
\[
\setm{N_s^A }{ \exists\alpha < \kappa\ (f''\alpha = \dom s)},
\]
where \( N^A_s =\setm{x\in\pre{A}{\kappa}}{s\subseteq x} \).
\end{definition}

 We briefly recall some useful applications of Definition~\ref{Fact : bounded topology}.

\begin{enumerate-(a)}
\item If \( G \) is a group of cardinality \( \kappa \) then, we define the \emph{\( \kappa \)-space of subgroups of \( G\)} by identifying
each subgroup of \( G \) with its characteristic function and setting
\[
\mathrm{Sub}_{G}=\setm{H\in\pre{G}{2}}{1_G\in H\wedge\forall x,y\in G\ (x,y\in H \to xy^{-1}\in H)},
\] which is a closed subset of
\( \pre{G}{2} \) and therefore is standard Borel.

\item Fix a language  consisting of finitary relation symbols \(L=\setm{R_i}{ i \in I}$, $|I| < \kappa\), and let \( n_i \) be the arity of \( R_i \).
We denote by \( X_L^\kappa \) \emph{the \(\kappa\)-space of \( L \)-structures with domain \( \kappa \)}. Every \( \mathcal A \in X_L^\kappa \) is a pair
\( (\kappa, \setm{R_i^\mathcal{A} }{ i \in I}) \) where each \( R_i^\mathcal{A} \) is an \( n_i \)-ary relation on \( \kappa \), so
it can be identified with an element of \( \prod_{i\in I} \pre{(\pre{n_i}{\kappa})}{2} \) in the obvious way. It follows that \( X_L^\kappa \) can be endowed with the product of the bounded topologies on its factors \( \pre{(\pre{n_i}{\kappa})}{2} \).
\end{enumerate-(a)}

For an infinite cardinal \( \kappa \), we consider the infinitary logic
\( L_{\kappa^+\kappa} \). In such logic formulas are defined inductively
with the usual formation rules for terms, atomic formulas, negations, disjunctions and conjunctions of size \( \leq \kappa \), and quantifications over less than \( \kappa \) many variables.

\begin{definition}\label{definition : space of models}
Given an infinite cardinal \( \kappa \) and an \( L_{\kappa^{+}\kappa} \)-sentence \( \phi\), we define \emph{the \(\kappa\)-space of \( \kappa \)-sized models of $\phi\)} by
\begin{equation*}
X^\kappa_\phi \coloneqq \setm{\mathcal{A}\in X_L^\kappa}{\mathcal{A}\models \phi}.
\end{equation*}
\end{definition}
The following theorem is a generalization of a classical result by L\'opez-Escobar for spaces of uncountable structures.

\begin{theorem}[$\kappa^{<\kappa} = \kappa$]\label{Theorem : LopezEscobar}
A set $B \subseteq X^\kappa_L$ is Borel and closed under isomorphism if and only if there is an $L_{\kappa^+\kappa}$-sentence $\phi$ such that $B = X^\kappa_\phi$.
\end{theorem}
To see a proof of Theorem \ref{Theorem : LopezEscobar} we refer the reader to \cite[Theorem~24]{FriHytKul} or \cite[Theorem~8.7]{AndMot}.
A straightforward consequence of it is that the
space defined in Definition \ref{definition : space of models} is standard Borel.

We conclude this subsection with a list of those spaces of models of \( L_{\kappa^{+}\kappa}\)-sentences that we will use in the ensuing sections.
We denote by \( X_\mathsf{GRAPHS}^\kappa\) the \emph{space of \( \kappa \)-sized graphs}. By graph we mean an undirected graph whose edge relation is irreflexive.
We denote by \(X_\mathsf{TFA}^\kappa\) the \emph{space of \(\kappa\)-sized torsion-free abelian groups}.
 For any fixed ring \( R \) such that \( |R|<\kappa \),
we denote by \( X^\kappa_{R\text{-}\mathsf{MOD}} \), the \emph{space of \( \kappa \)-sized \( R \)-modules}. Here observe that
every \( r\in R \) is regarded as a unary functional symbol, interpreted as the left scalar multiplication by \( r \).
The axioms of \( R \)-modules are the following
\begin{itemizenew}
\item \( \phi_\mathsf{AB} \), i.e., the first order formula defining abelian groups,
\item \( \forall x\forall y\ \big( r(x+y)=rx+ry \big) \),
\item \( \forall x\ \big( (r+q)x=rx+qx \big) \),
\item \( \forall x\ \big ( r(qx)=(rq)x \big )\),
\item \( 1x=x \).
\end{itemizenew}
Since \(r\) and \(q\) vary in \(R\), which has size \(<\kappa\), the formula defining the class of \(R\)-modules is a formula in the logic \(L_{\kappa^{+}\kappa}\).

\subsection{Borel reducibility}
Let \( X,Y \) be standard Borel \( \kappa \)-spaces and
 \(L \) a fixed language such that \(|L|<\kappa\).
Given 
 \( \mathcal{A},\mathcal{B}\in X^\kappa_L \), we say that \( \mathcal{A}\) \emph{is embeddable} into \(\mathcal{B}\), in symbols \(\mathcal A \embeds^{\kappa}_L \mathcal B\), if there is \( x\in \pre{\kappa}{\kappa} \) which realizes an isomorphism between \( \mathcal A\) and \( \mathcal{B} \restriction \Ima x \). As pointed out in \cite[Section 7.2.2]{AndMot} one can show directly that \( \embeds^{\kappa}_L \)  is the projection on \( X_L^\kappa\times X_L^\kappa\) of a closed subset of \(X_L^\kappa\times X_L^\kappa\times \pre{\kappa}{\kappa}\), therefore
 the quasi-order \(\embeds^{\kappa}_L \) of embeddability between \(\kappa\)-sized \(L\)-structures is \( \analytic \). 
 We denote by \( \embeds_\mathsf{Z}^{\kappa} \) the quasi-order of embeddability on the \(\kappa\)-space \( X^{\kappa}_\mathsf{Z}\), where \(\mathsf{Z}\in\set{\mathsf{GRAPHS},\mathsf{TFA},R\mathsf{\text{-}MOD}}\).
 
 Let \( P ,Q \) be \( \analytic \) quasi-orders on the standard Borel \(\kappa\)-spaces \( X \) and \( Y \), respectively.
Recall that \( P \) \emph{Borel reduces} to \( Q \), in symbols \( P\leq_B Q \), if and only if there is a \(\kappa^{+}\)-Borel function \( f\colon X\to Y\) such that
\(
\forall x_0x_1\in X\ (x_0\mathrel{P}x_1\Leftrightarrow {f(x_0)}\mathrel{Q}{f(x_1)}).
\)
Moreover, \( Q \) is a \emph{complete \( \analytic \) quasi-order} if for every \( \analytic \) quasi-order \( P \) on a standard Borel \( \kappa \)-space,
\( P \leq_B Q \). 
 Similarly, we say that  \( E \) is
a \emph{complete \( \analytic \) equivalence relation} if \(F \leq_B E \), for every \( \analytic \) equivalence relation \( F \).
Any quasi-order \( Q \) on \( X \)
 induces canonically an equivalence relation on \( X \), which is denoted by \( E_Q \), and defined by setting \(x \mathrel{E_Q} y \) if and only if
\( x\mathrel{Q}y \) and \(y\mathrel{Q}x \) for all \( x,y\in X\).
It can be easily verified that if \( Q \) is a complete \( \analytic \) quasi-order then \( E_Q \) is a complete \( \analytic \) equivalence relation.
The following theorem is a generalization of \cite[Theorem 3.1]{LouRos} to the embeddability relation on uncountable graphs.

\begin{theorem}[Mildenberger-Motto Ros~{\cite{MilMot}}]\label{Theorem : MilMot}
If \( \kappa \) is uncountable such that \( \kappa^{<\kappa} = \kappa \), then the relation of embeddability \( \embeds^\kappa_\mathsf{GRAPHS} \)
on the \(\kappa\)-space of \(\kappa\)-sized graphs is a complete \( \analytic \) quasi-order.
\end{theorem}

A first version of Theorem \ref{Theorem : MilMot} was obtained by Motto Ros in~\cite[Corollary 9.5]{Mot13} provided that \( \kappa \) is weakly compact.

\section{Existence of algebras with prescribed endomorphism ring}\label{section : 3}

Some crucial results that will be used in the next sections are about the existence of
algebras with prescribed endomorphism ring.

\subsection{\( S \)-completions}
We recall the basic definitions and some simple facts on \( S \)-completions following the treatise of \cite[Chapter 1]{GobTrl}.

\begin{notation}
 In the remainder of this section \( R \) is a commutative ring with unit \( 1 \) and \( S \) is a subset of \( R \setminus \set{0} \) containing \( 1 \) and 
closed under multiplication.
\end{notation}

\begin{definition}
We say that \( R \) is \emph{\( S \)-reduced} if \( \bigcap_{s\in S}{sR} = 0 \), and
\( R \) is \emph{\( S \)-torsion-free} if for all \( s \in S \) and \( r\in R \), \( sr = 0 \) implies \(r=0\). Further, we say that \( R \) is an \emph{\( S \)-ring} provided that \( R \) is both \( S \)-reduced and \( S \)-torsion-free.
\end{definition}

In most of the applications \( S \) is assumed to be countable. In such case, we  denote \( S \) by \( \mathbb S\) as in \cite[Chapter 1]{GobTrl}.
Examples of \( \mathbb{S} \)-rings include every noetherian domain \( R \) with \( \mathbb{S} = \setm{a^n}{n \in \omega } \), for any \( a \in R \) such that \( a R \) is a proper principal ideal of \( R \) (see \cite[Corollary 1.3]{GobTrl}).

\begin{definition}
Let \( R \) be an \( S \)-ring and \( M \) be an \( R \)-module.
We say that  \( M \) is \emph{ \(S\)-reduced } if \(\bigcap_{s\in S}sM=0\), and \( M \) is
\emph{ \( S \)-torsion-free } if for every \( s \in S \) and \( m\in M \), \( sm = 0 \) implies \(m=0\).
\end{definition}

We denote by \( \widehat M \) the \emph{\( S \)-completion} of \( M \), which is defined as follows. Given \( s,q\in S \), we write
\( q \preceq s \) if there is \( t \in S \) such that
\( s = qt \). Then we set
\[
\widehat M \coloneqq  \varprojlim_{s\in S} M/sM,
\]
the inverse limit of inverse system of \( R \)-modules
 \( (\set{M/sM}_{s\in S},\set{\pi^s_q}_{q\preceq s} ) \), where
\[
\pi^s_q \colon M/sM \to M/qM,\qquad m+sM \mapsto m+qM.
\]

Any \( R \)-module can be given the natural \emph{linear \( S \)-topology}, i.e., the one generated by \( \setm{sM}{s\in S}\) as a basis of neighborhoods of \( 0 \). A \emph{Cauchy net} in \(M\) is a sequence \(( m_{s}\mid s\in S)\) taking values in \(M\), and such that \(m_{q}-m_{qs}\in qM\), for all \(q,s\in S\). We say that the Cauchy net \(( m_{s}\mid s\in S)\) has \emph{limit} \(m\in M\) if and only if \({m-m_{s}}\in sM\), for every \(s\in S\). Finally, we say that \(M\) is \emph{\(S\)-complete} if it is complete with respect to the \(S\)-topology, that is, every Cauchy net in \(M\) has a unique limit.
If \( M \) is \( S \)-reduced and \( S \)-torsion-free, then \( \widehat M \) is \( S \)-reduced, \( S \)-torsion-free and
\(S\)complete (see~\cite[Lemma 1.6]{GobTrl}).

The canonical map 
\[
 \eta_{M}\colon M\to \widehat{M}, \qquad m\mapsto (m+sM \mid s\in S),
 \]
 which is always a homomorphism of \( R \)-modules,
 is injective if and only if \( M \) is \( S \)-reduced;
 and it is a ring homomorphism, whenever \( R=M \). Moreover, if \( M \) is \( S \)-complete, then \( \eta_{M} \) is an isomorphism. 

For any \( R \)-module \( M \), its \( S \)-completion \(\widehat M \) carries a natural a \( \widehat R\)-module structure. I.e., given \( \bar r=(r_{s}+sR \mid s\in S) \in \widehat R  \) and
\(\bar m=(m_{s}+sM \mid s\in S) \in \widehat M  \), we define the scalar multiplication by 
\[
\bar r \bar m\coloneqq (r_{s}m_{s}+sM \mid s\in S).
\]

\subsection{Existence of abelian groups with prescribed endomorphism ring}
If \( G \) is a \( \mathbb{Z}\)-module (i.e., an  abelian group) and \( \mathbb{S}=\mathbb{N} \setminus \set{0} \), then the \( \mathbb S \)-topology on \( G \) is usually called \( \mathbb{Z} \)-adic topology and the \( \mathbb S \)-completion of \( G \) is called the \emph{\( \mathbb{Z} \)-adic completion}. We refer to \cite[Theorem 39.5]{Fuc70} and \cite[Section 2.7]{Fuc15} as comprehensive sources on \( \mathbb{Z} \)-adic completions.

The next theorem was pointed out by Prze\'zdziecki in \cite{Prz14}. It states a slightly different result of a classical theorem by Corner~\cite[Theorem~A]{Cor63}, whose proof can be adapted to show the following.
\begin{theorem}[Prze\'zdziecki~{\cite[Theorem 2.3]{Prz14}}]\label{Theorem : Corner continuum}
Let $A$ be a ring of cardinality at most \( 2^{\aleph_0} \) such that its additive group is free. Then, there is a torsion-free abelian group \( M \subseteq \widehat A \) such that
\begin{enumerate-(i)}\label{enumerate : Corner theorem}
\item \label{condition : Corner1} \( A\subseteq M \) as (left) \( A \)-modules,
\item \label{condition : Corner2} \( \End{M}\cong A \),
\item \( |A| = |M| \).
\end{enumerate-(i)}
\end{theorem}

A few comments on Theorem \ref{enumerate : Corner theorem} may be of some help. We stress the fact that \( M \) is torsion-free. By construction \( M \) inherits the natural (left) \( A \)-module structure from \( \widehat A \),
which is the one defined by setting 
\begin{equation}\label{equation : M A-algebra}
a\ast m = \eta_{A}(a) m,
\end{equation}
for every \( a\in A \) and \( m\in M \).
Moreover, condition~\ref{condition : Corner2} of Theorem \ref{Theorem : Corner continuum} is proved by showing that for every \( h\in\End{M} \), there exists \( a\in A \) such that \(h(m)=a\ast m\) for all \(m\in M\).

\begin{remark}
The reader familiar with the bibliography may find our notation nonstandard.
In module theory people usually consider endomorphisms of \( R \)-modules as acting on the opposite side from the scalars (e.g., see~\cite[Chapter 1]{GobTrl}).
Nevertheless, we prefer to follow the notation of \cite{Prz14} and to have endomorphisms of \( M \) acting on the left. \end{remark}

\subsection{Existence of \(R\)-modules with prescribed endomorphism ring}

An \( R \)-module \( M \) is \emph{\( \mathbb{S} \)-cotorsion-free} if it is \( \mathbb S \)-reduced and \(\Hom{\widehat R}{M}=0 \).
This definition extends naturally to any \( R \)-algebra \( A \) by saying that \( A \) is \( \mathbb{S} \)-cotorsion-free if \(A\) has this property as an \(R\)-module.
It is shown in~\cite[Corollary 1.26]{GobTrl} that whenever
\( M \) is an \( R \)-module of size \( < 2^{\aleph_{0}}\), then \( M \) is \( \mathbb{S} \)-cotorsion-free if and only if it is \( \mathbb{S} \)-torsion-free and \( \mathbb{S} \)-reduced. We recall one more theorem of existence of \( R \)-algebras with prescribed endomorphism ring, that is a result of the same kind of Theorem~\ref{Theorem : Corner continuum}.

\begin{theorem}[G\"obel-Prze\'zdziecki~{\cite[Corollary 4.5]{GobPrz}}]
\label{prop : GobTrl}
Let \( R \) be an \( \mathbb{S} \)-cotorsion-free ring
such that \(|R|< 2^{\aleph_{0}}\).
If \( A \) is an \(R\)-algebra of cardinality at most the continuum with a free additive structure over  \(R\), 
then there exists an \( R \)-module \( M \) such that:
\begin{enumerate-(i)}
\item  \( A\subseteq M \subseteq \widehat A \) as (left) \( A \)-modules,
\item  \( \End_{R}{M}\cong A \),
\item \( |A| = |M| \).
\end{enumerate-(i)}
\end{theorem}

\section{The embeddability relation between torsion-free abelian groups}
\label{section : 4}

In this section we focus on the embeddability relation between \(\kappa\)-sized torsion-free abelian groups and we prove Theorem~\ref{Theorem : main}. To this purpose we adapt the embedding from the category of graphs into the category of abelian groups defined in~\cite{Prz14}.
For the sake
of exposition we avoid the notion of colimit commonly used in category theory, instead we use the classical notion of direct limit for direct systems of abelian groups
which gives more insights
on the possibility to define the reduction in a \(\kappa^{+}\)-Borel way.

Let \( \Gamma \) be a skeleton of the category of countable graphs;
i.e., a full subcategory of the category of countable graphs with exactly one object for every isomorphism class. Without loss of generality, assume that every object in \( \Gamma \) is a graph over a subset of \( \omega \).
Since we work under the assumption \eqref{equation : k<kequalsk}, there is
a \emph{\( \kappa\)-sized universal graph}, i.e., a graph of cardinality \( \kappa \), which contains all graphs of cardinality \( \kappa \) as induced subgraphs.
We denote by \( W_\kappa \) the \( \kappa\)-sized universal graph on \( \kappa \) and by \( [W_{\kappa}]^{\kappa} \) the subspace of induced subgraphs of \(W_{\kappa}\) of cardinality \( \kappa \). We identify \([W_{\kappa}]^{\kappa}\) with the \(\kappa^{+}\)-Borel subset of subsets of \(W_{\kappa}\) of cardinality \(\kappa\).
Therefore we can consider \( [W_{\kappa}]^{\kappa} \) as the standard Borel \(\kappa\)-space of graphs of cardinality \( \kappa \).

For every graph \( T \) and every infinite cardinal \( \lambda \), we denote by \( [T]^{<\lambda} \) the set of induced subgraphs of \( T \) of cardinality \( <\lambda \).
Next, for every \( S\in[W_\kappa]^{<\omega_1} \), we fix an isomorphism \( \theta_S\colon S\to \sigma(S) \), where \( \sigma(S) \) denotes the unique graph in $\Gamma$ which is isomorphic to \( S \).

Now define 
\begin{equation}\label{eq : defA}
 A \coloneqq \mathbb{Z}\big[ \Arw{\Gamma}\cup\set{1}\cup\pow_{fin}{(\omega)} \big].
 \end{equation}
That is, the free abelian group generated by the arrows in \( \Gamma \), a distinguished element \( 1 \), and the finite subsets of \( \omega \). We endow \( A \) with a ring structure by multiplying the elements of the basis as follows,
 and then extending the multiplication to the whole \( A \) by linearity.
For every \( a,b\in \Arw{\Gamma}\cup\pow_{fin}{(\omega)} \) let
\begin{align}\label{eq : multiplication}
ab&=\begin{cases}a\circ b & \text{if\quad}\begin{cases}a,b\in\Arw{\Gamma}\\\text{\(a\) and \(b\) are composable}\end{cases}\\
a''b & \text{if\quad}\begin{cases}b\subseteq\dom a\\ a\restriction b\text{ is an isomorphism}\end{cases}\\
0 & \text{otherwise}
\end{cases}\\
a1&=1a=a.
\end{align}

\begin{remark}
The definition of \( A \) in \eqref{eq : defA} differs from the one of \cite[Section 3]{Prz14} for including
\( \pow_{fin}{(\omega)} \) in the generating set. These elements will play the crucial role of embeddability detectors in Lemma \ref{Lemma : G reduction}. 
\end{remark}

Now observe that the ring \( A \) has cardinality \( 2^{\aleph_0} \) and its additive group is free.  So let 
\( M \) be a group having endomorphism ring isomorphic to \( A \) as in
Theorem~\ref{Theorem : Corner continuum}.
Notice that the elements of \( A \) act on \( M \) on the left as in \eqref{equation : M A-algebra}.

\begin{definition}
For every $C\in \Gamma$, let
\[
G_C\coloneqq id_C\ast M.
\]

Notice that \(G_{C}\) is a subgroup of \(M\), for all \(C\in \Gamma\).
 Moreover, if \( C,D\in \Gamma \) and \( \gamma\colon C\to D \), then \( \gamma\in A \) and thus induces a group homomorphism \( G\gamma \) from \( G_C\) to \( G_D \) by left-multiplication
\begin{equation}\label{eq : Ggamma}
G\gamma\colon G_C \to G_D\qquad
id_C \ast m \mapsto \gamma \ast(id_C\ast m).
\end{equation}
We make sure that such map is well defined as
\({\gamma \ast (id_C \ast m)}={id_D\ast(\gamma \ast( id_C\ast m))}\), which is clearly an element of \( G_D \).
\end{definition}

Now fix any \( T\in [W_{\kappa}]^{\kappa} \).
For every $S,S'\in[T]^{<\omega_1}$ such that $S\subseteq S'$, the inclusion map
$i^S_{S'}\colon S\to S'$ induces a map $\gamma^S_{S'}$ from $\sigma(S)$ to $\sigma(S')$, the one that makes the 
diagram below commute.
\begin{center}
\begin{tikzpicture}
  \matrix (m) [matrix of math nodes,row sep=3em,column sep=4em,minimum width=2em]
  { S& S' \\
     \sigma(S)&\sigma(S')\\};
  \path[-stealth]
    (m-1-1) edge node [above] {\(i^S_{S'}\)} (m-1-2)
    		edge node [left] {\(\theta_S\)} (m-2-1)
	(m-1-2) edge node [right] {\(\theta_{S'}\)} (m-2-2)
	(m-2-1) edge node [dashed, below] {\(\gamma^S_{S'}\)} (m-2-2);	
\end{tikzpicture}
\end{center}
The map \( \gamma^S_{S'} \) is in \( \Gamma \) and thus it induces functorially a group homomorphism \( G\gamma^S_{S'} \) as described in \eqref{eq : Ggamma}.
For all \( S,S'\in[T]^{<\omega_1} \) such that $S\subseteq S'$, let
\( \tau^S_{S'} = G\gamma^S_{S'}\).
We claim that
\((\set{G_{\sigma(S)}},\set{\tau^S_{S'}}_{S\subseteq S'} \big)_{S,S'\in[T]^{<\omega_1}}\)  is a direct system of torsion-free abelian groups indices by the poset \( [T]^{<\omega_1} \), which is ordered by inclusion.

\begin{definition} \label{Definition : the reduction}
For every \( T\in [W_{\kappa}]^{\kappa} \), let\footnote{As the referee kindly pointed out the notation \(GT\) instead of \(G(T)\) or \(G_{T}\) is nonstandard in descriptive set theory. Nevertheless we prefer to stick to the common practice in category theory to denote functors by juxtaposition.}
\begin{equation}\label{eq : the reduction}
GT\coloneqq\varinjlim_{S\in [T]^{<\omega_1}} G_{\sigma(S)}.
\end{equation}
\end{definition}

For the sake of definiteness, every element of the direct limit in \eqref{eq : the reduction} is regarded as the equivalence class \([(m,S)]\) of an element of the disjoint union \( \bigsqcup_{S\in [T]^{<\omega_1}} G_{\sigma(S)} \) factored out by the equivalence relation \( \sim_{T} \), which is defined by setting
\( (m,S)\sim_{T}(m',S') \) provided that there is \( S'' \supseteq S,S' \) such that
\( \tau^{S}_{S''}(m) = \tau^{S'}_{S''}(m') \). Such characterization for~\eqref{eq : the reduction} holds because the poset of indexes is directed (see~\cite[Corollary 5.31]{Rot}).

Notice that for every \(T\in [W_{\kappa}]^{\kappa}\), the group \( GT \) is abelian by definition, and it is torsion-free as torsion-freeness is preserved by taking subgroups and colimits. Moreover, we claim that \( GT \) has cardinality \( \kappa \). It is clear that \( |GT| \) is bounded by \( |\bigsqcup_{\kappa}M|=\kappa \). Further, we observe that each \(GT\) has at least \( \kappa \) distinct elements.
To see this, consider \( id_{C_0} \), where \(C_0\) stands for the unique graph with one vertex and no edges in \( \Gamma \). For sake of definiteness, suppose that
\( C_0 \) is the graph with no edges whose unique vertex is \( 0 \). For every \( \alpha<\kappa \),  let \( \set{\alpha} \) denote the subgraph of \( T \) with the only vertex \( \alpha\). It is clear that \( id_{C_0}\in G_{\sigma(\set{\alpha})} \). Moreover, for any distinct \( \alpha,\beta\in \kappa \), we have that \( (id_{C_0}, \set{\alpha}) \) and \( (id_{C_0}, \set{\beta}) \) represent two distinct elements of \( G_T \).
For, if \( S\supseteq \set{\alpha},\set{\beta} \), then
one has 
\begin{align*}
\tau^{\set{\alpha}}_S(id_{C_0})&=\gamma^{\set{\alpha}}_{S}id_{C_0}=(0\mapsto \theta_S(\alpha))\\
\tau^{\set{\beta}}_A(id_{C_0})&=\gamma^{\set{\beta}}_{S}id_{C_0}=(0\mapsto\theta_S(\beta)),
\end{align*}
which are not equal as $\theta_S$ is bijective.

The next lemma basically states that \( G \) can be defined in a \(\kappa^{+}\)-Borel way.

\begin{lemma}\label{Lemma : G Borel}
There is a \(\kappa^{+}\)-Borel map
\[
[W_{\kappa}]^{\kappa}\to X_\mathsf{TFA}^{\kappa}\qquad T\mapsto \mathcal{G}T
\] 
such that, for every  \( T\in [W_{\kappa}]^{\kappa}\), the group \( \mathcal{G}T \) is isomorphic to \( GT \).
\end{lemma}
\begin{proof}
Let \( \preceq \) be a well-ordering of \( B= \bigsqcup_{S\in[W_{\kappa}]^{<\omega_{1}}} G_{\sigma(S)} \).
First consider the map
\[
f\colon[W_{\kappa}]^{\kappa}\to 2^B,\qquad
T\mapsto \bigsqcup_{S\in[T]^{<\omega_{1}}} G_{\sigma(S)}.
\]
To see that \( f \) is \(\kappa^{+}\)-Borel consider the subbasis of \( 2^{B} \) given by the sets
\( \setm{x\colon B \to 2}{x((m,S))=1} \) and \( \setm{x\colon B \to 2}{x((m,S))=0} \),
for every \( (m,S)\in B\).
For any fixed \( (m_{0},S_{0}) \in B\), one has
\[
f^{-1}(\setm{x\colon B \to 2}{x((m_{0},S_{0}))=1})=\setm{T\in [W_{\kappa}]^{\kappa}}{S\subseteq T}
\] which is \(\kappa^{+}\)-Borel.

Then let \(g\colon \Ima{ f }\to 2^{B}\) be the map defined by mapping
\( f(T) \) to the subset of \( f(T) \) which is obtained by deleting all of the \( (m,S) \) that are \( \sim _{T} \)-equivalent (i.e., equivalent in the relation used to define the direct limit indexed by \( [T]^{<\omega_{1}}\)) to some point appearing before in the well-ordering \( \preceq \).
One has 
\[
g(f(T)) ((m,S))=1 \iff S\subseteq T \wedge \forall (m',S')\prec(m,S)((m',S')\nsim_{T} (m,S))
\]
where \( (m',S')\nsim_{T} (m,S) \) is a shorthand for
\[ \nexists S''\supseteq S,S'' (\tau^{S}_{S''}(m)=\tau^{S'}_{S''}(m')). \]

Then, for every \( T \), we define a group \( \mathcal{G} T\) with underlying set \( \kappa \) and operation \(\star_{T}\) by setting
\( \alpha\star_{T}\beta=\gamma \) if and only if the product of the \(\alpha\)-th element and the \( \beta \)-th element in \( g(f(T)) \) according to \( \preceq \) is \( \sim_{T} \)-equivalent to the \(\gamma\)-th element in \( g(f(T)) \). Notice that there is a unique element in \( g(f(T)) \) which is \(\sim_{T}\)-equivalent to such product, thus
the map \(T\mapsto \mathcal{G}T \) is well defined and is \(\kappa^{+}\)-Borel.
\end{proof}

Next lemma is derived essentially as in \cite[Lemma 3.6]{Prz14}.

\begin{lemma}\label{Lemma : G homomorphism}
If \( T,V\in X^\kappa_\mathsf{GRAPHS} \) and  \(T\embeds_\mathsf{GRAPHS}^\kappa V\), then \(GT\embeds_\mathsf{TFA}^\kappa GV\).
\end{lemma}
\begin{proof}
We first claim that if \(C,D\in\Gamma \) and \(\gamma\colon C\to D \) is an embedding then
\(G\gamma\colon G_C\to G_D\) is one-to-one.
Notice that by \ref{condition : Corner1} of Theorem~\ref{Theorem : Corner continuum} and the definition of \( G_C \)
one obtains
\[
{\mathbb{Z}[\Gamma_C\cup\mathcal{P}_{fin}(C)]} \subseteq
G_C \subseteq
\reallywidehat{\mathbb{Z}[\Gamma_C\cup\mathcal{P}_{fin}(C)]}.
\]
Acting by left-multiplication, $\gamma$ induces the injective map
\[
\langle\gamma\rangle\colon\mathbb{Z}[\Gamma_C\cup\mathcal{P}_{fin}(C)]\to \mathbb{Z}[\Gamma_D\cup\mathcal{P}_{fin}(D)],\qquad
a\mapsto \gamma a,
\]
which in turn induces the injective map on the $\mathbb{Z}$-adic completions
\begin{equation}\label{eq : hatgamma}
\widehat{\langle\gamma\rangle}\colon\reallywidehat{\mathbb{Z}[\Gamma_C\cup\mathcal{P}_{fin}(C)]}\to \reallywidehat{\mathbb{Z}[\Gamma_D\cup\mathcal{P}_{fin}(D)]},\qquad
\bar a\mapsto \gamma \ast \bar a.
\end{equation}
Comparing \eqref{eq : hatgamma} with \eqref{eq : Ggamma} it follows that \( G\gamma \) is indeed the restriction of \( \widehat{\langle\gamma\rangle} \) on \( G_C \), which implies that \( G\gamma \) is injective because so is \( \widehat{\langle\gamma\rangle} \).

Now let \( \phi\colon T\to V \) be a graph embedding. Then
there exists a group homomorphism
\[
G\phi\colon GT\to GV,\qquad [(g, S)]\mapsto [(G\gamma^S_{\phi''S}(g),\phi''S)],
\]
where \( \phi''S \) is the point-wise image of \( S \) through \( \phi \) and
 \( \gamma^S_{\phi''S}\colon\sigma(S)\to \sigma(\phi''S) \) is the map induced by \( \phi\restriction S \), which is clearly a graph embedding. We are left to prove that \( G\phi \) is one-to-one.
So fix any \( [(g,S)],[(g',S')]\in GT\) such that \( [(g,S)]\neq[(g',S')]\).
By directedness of $[T]^{<\omega_1}$ we can assume that $S=S'$ without any loss of generality. One has
\begin{align*}
G\phi([(g,S)])&=[(G\gamma^S_{\phi''S} (g),\phi''S)],\\
G\phi([(g',S)])&=[(G\gamma^S_{\phi''S} (g'),\phi''S)],
\end{align*}
which are different elements of \( GV \) because \( G\gamma^S_{\phi''S} \) is injective.
\end{proof}

Now we are left to prove that \( G T\embeds^{\kappa}_\mathsf{TFA} GV\) implies that \(T\embeds_\mathsf{GRAPHS}^{\kappa} V \). Given any linear combination \( \sum k_i\phi_i \), \( k_i\in \mathbb{Z}\) and \( \phi_i\in\Hom{T}{V} \), one can define a group homomorphism \(\Psi(\sum k_i\phi_i ) \colon GT\to GV \) as follows.
For any $\phi_i$ and $S\in[T]^{<\omega_1}$,
let \( \delta_i^S \) be the function such that the diagram commutes

\begin{center}
\begin{tikzpicture}
  \matrix (m) [matrix of math nodes,row sep=3em,column sep=4em,minimum width=2em]
  {
     S& \phi_i''S \\
     \sigma(S)&\sigma(\phi_i''S)\\};
  \path[-stealth]
    (m-1-1) edge node [above] {\(\phi_i\restriction S\)} (m-1-2)
    		edge node [left] {\(\theta_S\)} (m-2-1)
	(m-1-2) edge node [right] {\(\theta_{\phi_i''S}\)} (m-2-2)
	(m-2-1) edge node [dashed, below] {\(\delta^S_i\)} (m-2-2);	
\end{tikzpicture}
\end{center}
Since $\delta^S_i$ is an arrow in $\Gamma$, it induces a group homomorphism
\begin{align*}
G\delta^S_i\colon G_{\sigma(S)} \to G_{\sigma(\phi_{i}''S)}\qquad
m \mapsto \delta^S_i\ast m.
\end{align*}
 as observed  in \eqref{eq : Ggamma}. Then we define
\[
\Psi(\sum k_i\phi_i )\colon GT\to GV\qquad
[(m,S)]\mapsto \sum k_i[(G\delta^S_i(m), \phi_i'' S)].
\]

\begin{theorem}[Prze\'zdziecki~{\cite[Theorem 3.14]{Prz14}}]\label{Theorem : almost-fullness}
There is a natural isomorphism
\[
\Psi\colon\mathbb{Z}[\Hom{T}{V}]\xrightarrow{\cong} \Hom{GT}{GV}.
\]
\end{theorem}

\begin{remark}
Theorem~\ref{Theorem : almost-fullness} states that \( G \) is an almost-full embedding, according to the terminology of~\cite{Prz14,GobPrz}. It can be proved arguing as in~\cite[Section 3]{Prz14}.
\end{remark}

Now we come to the point where our modification from~\eqref{eq : multiplication} becomes crucial. Since $A$ contains the finite subsets of $\omega$, we use them and the property of almost-fullness of \( G \) to detect an embedding among
\( \phi_{0},\dotsc,\phi_{n} \) when \( \Psi(\sum_{i\leq n} k_{i}\phi_{i}) \) is one-to-one.

\begin{lemma}\label{Lemma : G reduction}
For every two graphs
\(T\) and \(V\) in \( X^\kappa_\mathsf{GRAPHS} \), if \( GT\embeds^{\kappa}_\mathsf{GROUPS} GV \) holds then \( T\embeds^{\kappa}_\mathsf{GRAPHS} V \).
\end{lemma}
\begin{proof}Let $T,V$ be as in the hypothesis and \( h\colon GT\to GV \) a group embedding.
By Theorem \ref{Theorem : almost-fullness} we have
\[
h=\Psi(\sum_{i\in I} k_i\phi_i),
\]
for some linear combination of
graph homomorphisms \(\phi_{i}\in \Hom{T}{V} \).
We claim that there must be some $i\in I$ such that $\phi_i$ is a graph embedding from $T$ into $V$. Suppose that it is not true, aiming for a contradiction. Since \( [T]^{<\omega_1} \) is directed, there is some finite \( S\in[T]^{<\omega} \) such that, for every \( i\in I \), the restriction map \( \phi_i\restriction S \) is not one-to-one or does not preserve non-edges.
Call \( d \) the vertex set of \( \sigma(S) \). Such \( d \) is a finite subset of \(\omega \) and is an element of \( G_{\sigma(S)} \) because \( d=id_{\sigma(S)}\ast d \).
Now consider $[(d,S)]$, the element of  \( GT \) represented by  \( d\in G_{\sigma(S)} \). Then $[(d,S)]$ is a  nontrivial element and
\begin{align*}
h([(d,S)])=&\sum k_i[(G\delta^S_i (d), \phi_i'' S)]=\\
	=&\sum k_i[(\delta^S_i\ast d,\phi_i'' S)]=0\\
\end{align*}
because if $\phi_i\restriction S$ is not an embedding then neither is the induced
map $\delta^S_i$. This contradicts the fact that $h$ is one-to-one.
\end{proof}

Summing up the results of this section we can prove the main theorem.

\begin{proof}[Proof of Theorem \ref{Theorem : main}]
In view of the Lemma \ref{Lemma : G Borel}, we can assume that $G$ is \(\kappa^{+}\)-Borel. By Lemma~\ref{Lemma : G homomorphism}, \(G\) is a homomorphism from $\embeds_\mathsf{GRAPHS}^\kappa$ to $\embeds_\mathsf{TFA}^\kappa$, and
Lemma \ref{Lemma : G reduction} yields that $G$ is a reduction.
\end{proof}

\begin{corollary}\label{Proposition : reduction GR->TFA} For every uncountable \(\kappa\) such that \(\kappa^{<\kappa}=\kappa\), the embeddability relation \( \embeds_\mathsf{TFA}^{\kappa} \) on the \( \kappa \)-space of \(\kappa\)-sized torsion-free abelian groups is a complete \( \analytic \) quasi-order.
\end{corollary}
\begin{proof}
Combining Theorem \ref{Theorem : main} with Theorem \ref{Theorem : MilMot}, it follows that
\( \embeds_\mathsf{TFA}^{\kappa} \) is a complete \( \analytic \) quasi-order provided that \( \kappa^{<\kappa}=\kappa \) holds.
\end{proof}

It is worth mentioning that the analogue of Corollary~\ref{Proposition : reduction GR->TFA}, where \(\kappa=\omega\), has been proved recently in~\cite{CalTho}.


\section{The embeddability relation between \( R \)-modules}
\label{Section5}

In this section we use the second Corner's type theorem stated in section \ref{section : 3} to prove that the quasi-order of embeddability between \( R \)-modules, for any \( \mathbb{S} \)-ring \( R \) of cardinality less than the continuum, is a complete \(\analytic \) quasi-order.

\begin{proof}[Proof of Theorem~\ref{Theorem : R-mod}]
Let \( \Gamma \), \(W_{\kappa} \), and \( \sigma \) be as in section \ref{section : 4}. Define
\begin{equation}\label{eq : defA-R}
 A \coloneqq R\big[ \Arw{\Gamma}\cup\set{1}\cup\pow_{fin}{(\omega)} \big].
 \end{equation}
That is, the free \( R \)-module generated by the arrows in \( \Gamma \), a distinguished element \( 1 \) and the finite subsets of \( \omega \). The \( R \)-module \( A \)
can be endowed a ring structure by defining a multiplication on the element of its basis as in \eqref{eq : multiplication}. Such multiplication is 
compatible with the \( R \)-module structure, therefore we can regard \( A \) as an \( R \)-algebra.
Notice that \( A \) has cardinality the continuum so we apply Theorem~\ref{prop : GobTrl} which yields the existence of an \( R \)-module \( M\cong \End_{R}{A} \). We continue defining
\( G \) similarly to how we did in Section~\ref{section : 4}. That is, for every \( C \in \Gamma \),
\[
G_{C} \coloneqq id_{C}\ast M,
\]
and for every \( T\in [W_{\kappa}]^{\kappa} ,\) let
\[
G T\coloneqq \varinjlim_{S\in [T]^{<\omega_1}} G_{\sigma(S)}.
\]

Since \( |M|=|A| \), for every \( T\in [W_{\kappa}]^{\kappa} \), \( GT \) has size \( \kappa \).
Then one can argue as in Lemmas \ref{Lemma : G Borel}, \ref{Lemma : G homomorphism}, and \ref{Lemma : G reduction} to prove that \( G \) is a Borel reduction from
\( \embeds^{\kappa}_\mathsf{GRAPHS} \) to
\( \embeds^{\kappa}_{R\text{-}\mathsf{MOD}} \). In this case the almost-fullness for \( G \) was proved essentially in \cite[Theorem 3.16]{GobPrz} with the same argument used in \cite{Prz14}.
\end{proof}

\begin{corollary}
For every uncountable cardinal \(\kappa\) such that \(\kappa^{<\kappa}=\kappa\) and every ring \(R\) as in the statement of Theorem~\ref{Theorem : R-mod}, the embeddability relation  \( \embeds_{R\text{-}\mathsf{MOD}}^{\kappa} \) on the \( \kappa \)-space of \(\kappa\)-sized
\(R\)-module is a complete \( \analytic \) quasi order.
\end{corollary}

\section{The isomorphism problem}

At the end of his paper Prze\'zdziecki posed the question if every two isomorphic groups in the target of the functor have isomorphic inverse images  (see \cite[Section 8]{Prz14}).
We still do not know whether after our modification the answer is positive, i.e., whether
the map \( G \) defined in Definition \ref{Definition : the reduction} is a reduction for the isomorphisms.
Then we ask a more general question in terms of Borel reducibility.

\begin{question}\label{question : iso_{TFA}}
Is there any Borel reduction from isomorphism \( \cong^{\kappa}_{\mathsf{GRAPHS}} \)
on \(\kappa\)-sized graphs to isomorphism \( \cong^{\kappa}_{\mathsf{TFA}}\)
on \(\kappa\)-sized torsion-free abelian groups?
\end{question}

Question \ref{question : iso_{TFA}} is still open even in the case \( \kappa = \omega \), where a positive answer will yield that the relation \(\cong_{\mathsf{TFA}}\) of isomorphism on countable torsion-free abelian groups is
a maximum up to Borel reducibility among all the equivalence relations
induced by a Borel action of \(S_{\infty}\) on a standard Borel space. A remarkable result in such direction is one by Hjorth, who proved in \cite{Hjo02} that \( \cong_\mathsf{TFA} \) is not Borel. In fact, this was extended by Downey-Montalb\'an \cite{DowMon}, who showed that \( \cong_\mathsf{TFA} \) is complete \(\analytic\) as a set of pairs. We ought to mention that there are several results in classical Borel reducibility concerning the isomorphism relation on torsion-free abelian groups with finite rank. For every \( n< \omega \), denote by \(\cong_{\mathsf{TFA}_{n}}\) the isomorphism on countable torsion-free abelian group of rank \(n\).
An old result by Baer~\cite{Bae37} establishes that \( \cong_{\mathsf{TFA}_{1}}\) is essentially \( E_{0} \) (i.e., it is Borel bi-reducible with \( E_{0} \)). 
Moreover Thomas proved in \cite{Tho03} that for every \( n\geq 1\),
\( (\cong_{\mathsf{TFA}_{n}}) <_{B} (\cong_{\mathsf{TFA}_{n+1}})\).

Now we go back to Question \ref{question : iso_{TFA}}. Some results of \cite{FriHytKul} and \cite{HytMor} use certain model theoretic properties of complete theories to obtain information about the isomorphism relation between the models of those.
First let us mention that in \cite{HytKul} the authors give the following example, among many others, of a complete \( \analytic \) equivalence relation in the constructible universe.
\begin{definition}
Let \( E_{\omega}^{\kappa} \) the equivalence relation defined on \(\pre{\kappa}{\kappa}\)
by
\[
x\mathbin{E_{\omega}^{\kappa}} y \iff
\setm{\alpha<\kappa}{x(\alpha)=y(\alpha)}\text{ contains an \(\omega\)-club}.
\]
\end{definition}

In \cite[Theorem 7]{HytKul}, under the assumption \( \mathrm V= \mathrm L \), the equivalence
 relation \( E_{\omega}^{\kappa} \) is shown to be complete \(\analytic\) for every inaccessible cardinal \( \kappa \).
Then, in \cite [Definition 5.4]{HytMor} the authors defined the \emph{orthogonal chain property} (\( \OCP \)) for stable theories and proved the following result.

\begin{theorem}[Hyttinen-Moreno~{\cite[Corollary 5.10]{HytMor}}]\label{Theorem : HytMor}
Assume that \(\kappa\) is inaccessible. For every stable theory \( T \) with \(\OCP\), the equivalence relation
\( E_{\omega}^{\kappa} \) reduces continuously to \(\cong_{T}^{\kappa}\).
\end{theorem}

As it was kindly pointed out by Hyttinen to the author of this paper, the theory of \( \mathbb{Z}_{p} \) (i.e.,  the group of \( p \)-adic integers) has \(\OCP\) and is stable. Thus one obtains Theorem~\ref{Theorem : iso} as a corollary
of Theorem~\ref{Theorem : HytMor}. This gives an affirmative partial answer to
Question~\ref{question : iso_{TFA}}.


\end{document}